\documentclass[11pt]{article}

\usepackage{amsmath,amsfonts,amssymb,MnSymbol}

\newcommand{\R}{{\mathbb R}}
\newcommand{\N}{{\mathbb N}}

\newcommand{\Z}{{\mathbb Z}}

\usepackage{bbm}

\def\E{{\cal E}}

\newtheorem{defn}{Definition}[section]
\newtheorem{theo}[defn]{Theorem}

\newtheorem{rem}[defn]{Remark}
\newtheorem{exam}[defn]{Example}

\newenvironment{proof}{{\bf Proof }}{{\vskip 0.1cm \hfill$\Box$}}

\begin{document}

\noindent
{\Large\bf Explicit recurrence criteria for symmetric gradient type Dirichlet forms satisfying a Hamza type condition}

\bigskip
\noindent
{\bf by Minjung Gim} \footnote{The research of Minjung Gim was supported by NRF (National Research Foundation of Korea) Grant funded by the Korean Government(NRF-2012-Global Ph.D. Fellowship Program).} 
{\bf and Gerald Trutnau} \footnote{The research of Gerald Trutnau was supported by Basic Science Research Program through the National Research Foundation of Korea(NRF) funded by the Ministry of Education, Science and Technology(MEST)(2012006987) 
and Seoul National University Research Grant 0450-20110022.}

\bigskip
\noindent
{\it Department of Mathematical Sciences and Research Institute of Mathematics of Seoul National University,
599 Gwanak-Ro, Gwanak-Gu, Seoul 151-747, South Korea}

\bigskip
\noindent
{E-mails: ellipse@snu.ac.kr, trutnau@snu.ac.kr}

\bigskip
\noindent
{\small{\bf Abstract.}  In this note, we present explicit conditions for symmetric gradient type Dirichlet forms to be recurrent. 
This type of Dirichlet form is typically strongly local and hence associated to a diffusion. We consider the one dimensional case 
and the multidimensional case, as well as the case with reflecting boundary conditions. Our main achievement is that the explicit results are obtained under quite weak assumptions on the closability, hence regularity of the underlying coefficients.  Especially in dimension one, where a Hamza type condition is assumed, the construction of the sequence of functions $(u_n)_{n\in \N}$ in the Dirichlet space that determine 
recurrence works for quite general Dirichlet forms but is still explicit.}

\bigskip
\noindent
{\bf Mathematics Subject Classification (2000)}: 31C25, 60J60, 60G17.

\bigskip\noindent 
{\bf Key words:} Symmetric Dirichlet forms, diffusion processes, recurrence criteria.

\section{Introduction}\label{1}
For all notions and results that may not be defined or cited in this introduction we refer to \cite{FOT}. \\
Let $E$ be a locally compact separable metric space and $\mu$ be a $\sigma$-finite measure on the Borel $\sigma$-algebra of $E$. 
Consider a regular symmetric Dirichlet form $(\E,D(\E))$ on $L^2(E,\mu)$ with associated resolvent $(G_\alpha)_{\alpha > 0}$. 
Then it is well-known that there exists a Hunt process 
$\mathbb{M}=(\Omega,({\cal F}_t)_{t\ge 0},(X_t)_{t\ge 0}, (P_x)_{x\in E_\Delta})$ with lifetime $\zeta$ such that the resolvent 
$R_\alpha f(x)=E_{x}[\int_0^{\infty}e^{-\alpha t}f(X_t)dt]$ of $\mathbb{M}$ is a q.c. $\mu$-version of $G_\alpha f$ for all $\alpha  >0$, $f\in L^2(E,\mu)$, $f$ bounded (see\cite[Theorems 4.2.3 and 7.2.1]{FOT}). Here $E_x$ denotes the expectation w.r.t. $P_x$. Then $(\E,D(\E))$ is (called) {\it recurrent}, iff for 
any $f\in L^1(E,\mu)$ which is strictly positive, we have
\begin{eqnarray}\label{recurrent}
E_x[\int_0^{\infty}f(X_t)dt]=\infty \mbox{ for } \mu\mbox{-a.e. } x\in E.
\end{eqnarray}
The last is equivalent to the existence of a sequence $(u_n)_{n\in \N}\subset D(\E)$ with $0\leq u_n \leq 1$, $n\in \N$, $u_n \nearrow 1 $ as $n \rightarrow \infty$ $\mu$-a.e. and
$\lim_{n\rightarrow \infty}\E(u_n , u_n)=0$ (cf. \cite[Lemma 1.6.4 (ii) and Theorem 1.6.3]{FOT} and \cite[Corollary 2.4 (iii)]{o92r}). 
In particular, recurrence implies conservativeness, which means that $P_x(\zeta=\infty)=1$ for quasi every $x\in E$ (see \cite[Corollary 2.4]{o92r}). If the Dirichlet form $(\E,D(\E))$ is additionally {\it irreducible}, then some more refined recurrence statements than (\ref{recurrent}) can be made 
(cf. \cite[Chapter 4.7]{FOT2} and \cite{ge80}), and if the Dirichlet form $(\E,D(\E))$ satisfies the absolute continuity condition (cf. \cite[(4.2.9)]{FOT}), then the statements hold pointwise and not only for quasi every $x\in E$ (see \cite[Problem 4.6.3]{FOT}).\\
A sufficient condition for recurrence determined by the $\mu$-volume growth of balls is derived in \cite[Theorem 3]{Stu1} for energy forms, i.e. Dirichlet forms that are given by a carr\'e du champs. The condition is very general, but in some cases not explicit, i.e. difficult to verify (cf. e.g. \cite[Proposition 3.3. and Theorem 3.11]{oukrutr} where the determination of the $\mu$-volume of balls is tedious). Moreover, in \cite{Stu1} it is throughout assumed that the underlying Dirichlet form is irreducible, which we do not assume. In \cite{o92r} and \cite[Theorem 1.6.7]{FOT} the recurrence determining sequence $(u_n)_{n\in \N}\subset D(\E)$ is explicitly constructed, but the underlying Dirichlet form fulfills stronger (explicit) assumptions than in \cite{Stu1}, and moreover \cite{o92r} and \cite[Theorem 1.6.7]{FOT} concern only variants of the $(a_{ij})$-case. \\
This note concerns an intermediate approach, which resembles the construction in \cite[Theorem 1.6.7]{FOT}, but under quite more general assumptions. 
In particular, we consider reference measures that are different to the Lebesgue measure.
The conditions for the existence of a recurrence determining sequence $(u_n)_{n\in \N}$ given here are (as in \cite{FOT} and \cite{Stu1}) sufficient. 
However, in dimension one they turn out to be equivalent, whenever there exists a natural scale (see e.g. \cite[Lemma 3.1]{o92r} and \cite[Theorem 3.11]{oukrutr}).
\section{The one dimensional non-reflected case}\label{2}
Let $\varphi \in L^1_{loc}(\R,dx)$ with $\varphi>0$ $dx$-a.e. and let $\mu$ be the $\sigma$-finite measure defined by $d\mu=\varphi dx$. In particular, it then follows that $\mu$ is $\sigma$-finite and $\mu $ has full support, i.e.
$$\int_V \varphi(x) dx>0\qquad for\; all\; V \subset \R,\; V\;non-empty\;and\;open.$$
Let $\sigma$ be a (measurable) function such that $\sigma>0$ $dx$-a.e. and such that $\sigma \varphi \in L^1_{loc}(\R, dx)$. Consider the symmetric bilinear form
$$\E(f,g):= \frac{1}{2} \int_\R \sigma(x) f'(x)g'(x) \mu(dx), \qquad f,g\in C^\infty_0(\R)$$
on $L^2 (\R,\mu)$. Here $f'$ denotes the derivative of a function $f$ and $C_0^\infty (\R)$ denotes the set of all infinitely differentiable function on $\R$ with compact support. Let $U$ be the largest open set in $\R$ such that
$$
\frac{1}{\sigma\varphi}\in L^1_{loc}(U,dx)
$$
and assume
\begin{equation}\label{Hamza}
dx(\R \setminus U)=0.
\end{equation}
\centerline{}
Furthermore we suppose $(\E,C_0^\infty (\R ))$ is closable on $L^2(\R,d\mu)$. (By the results of (\cite[II, 2 a)]{mr}), if there is some open set $\tilde{U}\subset \R$ such that $1/\varphi \in L^1_{loc}(\tilde{U},dx)$ and $dx(\R \setminus(U\cap \tilde{U}))=0$, then $(\E,C_0^\infty(\R))$ is closable on $L^2(\R,\mu)$.)  Denote the closure by $(\E,D(\E))$. We want to present sufficient conditions for the recurrence of the symmetric Dirichlet form $(\E,D(\E))$. Here we recall $(\E,D(\E))$ is called recurrent (in the sense of \cite{FOT}) if there exists a sequence $(u_n)_{n\in \N}\subset D(\E)$ with $0\leq u_n \leq 1$, $n\in \N$, $u_n \nearrow 1 $ as $n \rightarrow \infty $ $dx$-a.e. and
$$ \lim_{n\rightarrow \infty}\E(u_n , u_n)=0.$$

\centerline{}

\begin{rem} 
Under quite weak regularity assumptions on $\sigma$ and $\varphi$ one can show using integration by parts that the generator $L$ corresponding to $(\E,D(\E))$ (for the definition of generator see \cite{FOT} or \cite{mr}) is given by
$$
Lf= \frac{\sigma}{2}f'' + \frac{1}{2}\Big ( \sigma ' + \sigma \frac{\varphi '}{\varphi} \Big) f',
$$
i.e. for $f\in D(L)\subset D(\E)$ we have
$$-\int_\R Lf \cdot g d\mu =\E(f,g).$$
In particular (assuming again that everything is sufficiently regular) choosing
$$\varphi(x)=\frac{1}{\sigma(x)}e^{\int_0^x  \frac{2b}{\sigma}(s)ds}$$
for some (freely chosen) function $b$, we get
$$Lf= \frac{\sigma}{2}f''+bf'.$$
Therefore our framework is suitable for the description of nearly any diffusion type operator in dimension one with concrete coefficients.
\end{rem}
\centerline{}
Since $U\subset \R$ is open, $U$ is the disjoint union (we use the symbol $\cupdot$ to denote this) of countably (finite or infinite) many open intervals.\\
There are five possible cases that we summarize in the following theorem.

\begin{theo}\label{nonreflected}
$(\E,D(\E))$ is recurrent if one of the following conditions holds:\\
(i) $U=(-\infty, \infty)$ and
$$\int_{-\infty}^0 \frac{1}{\sigma \varphi}(s) ds = \int_0^\infty \frac{1}{\sigma \varphi}(s) ds =\infty.$$
(ii) $U=(-\infty,a) \cupdot V \cupdot (b,\infty)$ where $V$ is some open set (so either $V$ is empty if $a=b$, or $V$ is non-empty if $a<b$) and
$$\int ^{a-1}_{-\infty} \frac{1}{\sigma \varphi}(s) ds = \int^\infty_{b+1}\frac{1}{\sigma \varphi}(s) ds =  \infty.$$
(iii) $U= \bigcup_{n\in \N} I_{-n} \cupdot V \cupdot \bigcup_{n\in \N}I_n$, $I_n=(x_n, x_{n+1})$, $x_n < x_{n+1}$, $I_{-n}=(x_{-n}, x_{-n+1})$, $x_{-n} < x_{-n+1}$, $n\in \N$, $V$ some open set, and
\begin{align*}
a_n&:= \int_{c_n}^{d_n} \frac{1}{\sigma \varphi}(s) ds \; \longrightarrow \infty \\
a_{-n}=b_n&:=\int^{c_{-n}}_{d_{-n}} \frac{1}{\sigma \varphi}(s) ds \; \longrightarrow \infty
\end{align*}
as $n \rightarrow \infty$, where
$$c_n= \frac{x_{n+1}+x_n}{2}, \quad n\in \Z \setminus \{0\},$$
and for $n\in \N$
$$
d_n:= \begin{cases}\;x_{n+1}-\frac{1}{n}\qquad \;\qquad &\;\;x_{n+1}-c_n>1,\\
\; x_{n+1}-\frac{x_{n+1}-c_n}{n} &\;\;x_{n+1}-c_n\leq 1,
\end{cases}
$$
$$
d_{-n}:= \begin{cases}\;x_{-n}+\frac{1}{n}\qquad \;\qquad &\;\;c_{-n}-x_{-n}>1,\\
\; x_{-n}+\frac{c_{-n}-x_{-n}}{n} &\;\;c_{-n}-x_{-n}\leq 1.
\end{cases}
$$
(iv) $U=\bigcup_{n\in \N}I_{-n}\cupdot V \cupdot (b,\infty),$ $I_{-n}=(x_{-n}, x_{-n+1}),$ $x_{-n}<x_{-n+1}$, $n\in \N,$ $V$ some open set,
$$ 
\int^\infty_{b+1}\frac{1}{\sigma \varphi}(s) ds =  \infty,
$$
and 
$$
a_{-n}=b_n:=\int^{c_{-n}}_{d_{-n}} \frac{1}{\sigma \varphi}(s) ds \; \longrightarrow \infty
$$
as $n \rightarrow \infty$ where for $n\in \N$
$$
c_{-n}= \frac{x_{-n}+x_{-n+1}}{2}
$$
and
$$
d_{-n}(x):= \begin{cases}\;x_{-n}+\frac{1}{n}\qquad \;\qquad &\;\;c_{-n}-x_{-n}>1,\\
\; x_{-n}+\frac{c_{-n}-x_{-n}}{n} &\;\;c_{-n}-x_{-n}\leq 1.
\end{cases}
$$
(v) $U=(-\infty,a)\cupdot V \cupdot \bigcup_{n\in \N} I_n$, $I_{n}=(x_{n}, x_{n+1}),$ $x_{n}<x_{n+1}$, $n\in \N,$ $V$ some open set,
$$
\int ^{a-1}_{-\infty} \frac{1}{\sigma \varphi}(s) ds=\infty,
$$
and 
$$ 
a_n:= \int_{c_n}^{d_n} \frac{1}{\sigma \varphi}(s) ds \; \longrightarrow \infty
$$
as $n\rightarrow \infty$ where for $n\in \N$
$$
c_n= \frac{x_{n+1}+x_n}{2}
$$
and 
$$
d_n:= \begin{cases}\;x_{n+1}-\frac{1}{n}\qquad \;\qquad &\;\;x_{n+1}-c_n>1,\\
\; x_{n+1}-\frac{x_{n+1}-c_n}{n} &\;\;x_{n+1}-c_n\leq 1.
\end{cases}
$$
\end{theo}
\centerline{}
\centerline{}
\begin{rem}\label{openset}
With the obvious modifications Theorem \ref{nonreflected} can be reformulated for Dirichlet forms that are given as the closure of
$$ \frac{1}{2}\int_V \sigma(x) f'(x) g'(x) \mu(dx), \qquad f,g\in C^\infty_0(V)$$
where $V$ is an arbitrary open and connected set in $\R$. We omit this here to avoid trivial complications.
\end{rem}
\centerline{}
\centerline{}
\begin{proof}
(i) By \cite[Theorem 1.6.3]{FOT}, it suffices to find a sequence $(u_n)_{n\in \N}\subset D(\E)$ with $0\leq u_n \leq 1$, $n\in \N$, $u_n \nearrow 1 $ as $n \rightarrow \infty $ and $ \lim_{n\rightarrow \infty}\E(u_n , u_n)=0.$ Define sequences $a_n$ and $b_n$ by
$$
a_n:= \int_0^n \frac{1}{\sigma \varphi}(s)ds, \qquad b_n:= \int_{-n}^0 \frac{1}{\sigma \varphi}(s) ds.
$$
Since $1/\sigma\varphi \in L^1_{loc}(\R,dx)$, these are well defined and converge to $\infty$. Let\\
$$
u_n(x):= \begin{cases}\; 1-\frac{1}{a_n}\int_0^x \frac{1}{\sigma\varphi}(t)dt\qquad \; \;&x\in[0,n],\\
\; 1-\frac{1}{b_n}\int_x^0 \frac{1}{\sigma \varphi}(t)dt&x\in[-n,0], \\
\;0&elsewhere.
\end{cases}
$$
For each $n\in \N$, $u_n$ has compact support and is bounded. Moreover, since $1/\sigma\varphi \in L^1_{loc}(\R,dx)$, $u_n(x)$ is differentiable at every Lebesgue point of $1/\sigma \varphi$, hence $dx$-a.e. Furthermore, we can easily check $u_n \nearrow 1$ $dx$-a.e. as $n \rightarrow \infty$ and $u_n'(x)$ exists $dx$-a.e. with
$$
u_n'(x)= \begin{cases}\; -\frac{1}{a_n} \frac{1}{\sigma\varphi}(x)\qquad \; \;&x\in[0,n],\\
 \;\frac{1}{b_n}\frac{1}{\sigma \varphi}(x)&x\in[-n,0], \\
\; 0&elsewhere.
\end{cases}
$$
Thus, it remains to show that $(u_n)_{n\in \N} \subset D(\E)$. Let $\eta$ be a standard mollifier on $\R$. Set $\eta_\epsilon(x)= \frac{1}{\epsilon} \eta(\frac{x}{\epsilon})$ so that $\int_\R \eta_\epsilon dx=1$ and so that the support of $\eta_\epsilon$ is in $(-\epsilon,\epsilon)$. Then $\eta_\epsilon*u_n\in C^\infty_0(\R)$ and $ (\eta_\epsilon*u_n)'=\eta_\epsilon*u_n',$ $n\in \N.$ We have
$$ (\eta_\epsilon*u_n(x)-u_n(x))= \int_\R [u_n(x-y)-u_n(x)] \eta_\epsilon(y) dy,$$
$$| \eta_\epsilon*u_n(x)-u_n(x)|^2 \leq \int_\R |u_n(x-y)-u_n(x)|^2 \eta_\epsilon(y) dy,$$
\begin{align*}
\int_\R (\eta_\epsilon*u_n (x)-u_n(x))^2 \varphi(x) dx &\leq \int_\R \int_\R |u_n(x-y)-u_n(x)|^2 \eta_\epsilon(y) dy \varphi(x)dx \\
&= \int_\R \int_\R |u_n(x-y) - u_n(x)|^2 \varphi(x)dx \eta_\epsilon(y) dy\\
&=\int_\R \int_\R |u_n(x-\epsilon y)-u_n(x)|^2 \varphi(x) dx \eta(y)dy.
\end{align*}
\centerline{}
If we let $g(y):= \int_\R |u_n(x-y)-u_n(x)|^2 \varphi(x) dx$, then $g(0)=0$ and $g$ is continuous and bounded. Thus by the Dominated Convergence Theorem $\lim_{\epsilon \rightarrow 0} \eta_\epsilon*u_n  = u_n$ in $L^2(\R,\mu)$. Furthermore for $0<\epsilon<1$
\begin{align*}
\E(\eta_\epsilon*u_n,\eta_\epsilon*u_n)&= \int_\R \Big (\eta_\epsilon*u_n'(x) \Big )^2 \sigma(x)\varphi(x) dx \\
&\leq \| u_n' \|^2_{L^1(\R,dx)} \int_{-n-1}^{n+1} \sigma(x)\varphi (x) dx
\end{align*}
Since $| \eta_\epsilon*u_n'(x)| \leq \|u_n'\|_{L^1(\R, dx)} < \infty$, we have
$$ \sup_{0<\epsilon<1} \E(\eta_\epsilon*u_n, \eta_\epsilon*u_n) < \infty.$$
Thus from \cite[Lemma 2.12]{mr}, $u_n \in D(\E)$ for all $n\in \N.$
\begin{align*}
\E(u_n,u_n) &= \frac{1}{2} \int_\R u_n'(x)  u_n'(x) \sigma(x)\varphi (x) dx\\
&=\frac{1}{2} \Big[ \frac{1}{a_n^2} \int_0^n \frac{1}{\sigma\varphi}(x) dx +\frac{1}{b_n^2}\int_{-n}^0 \frac{1}{\sigma \varphi}(x) dx \Big]\\
&= \frac{1}{2} \big[\frac{1}{a_n}+\frac{1}{b_n}\big].
\end{align*}
Therefore  $\lim_{n\rightarrow \infty}\E(u_n,u_n)= 0$. i.e. $(\E,D(\E))$ is recurrent. \\
\centerline{}
(ii) Let
$$
u_n(x):= \begin{cases}\;1 \qquad \qquad \qquad \qquad &x\in [a-1,b+1],\\
\; 1-\frac{1}{a_n}\int_{b+1}^x \frac{1}{\sigma \varphi}(t)dt&x\in[b+1,b+1+n],\\
\; 1-\frac{1}{b_n}\int_x^{a-1} \frac{1}{\sigma \varphi}(t)dt&x\in[a-1-n,a-1], \\
\;0&\quad elsewhere,
\end{cases}
$$
where $a_n=\int^{b+1+n}_{b+1} \frac{1}{\sigma\varphi}(s) ds$, $ b_n=\int^{a-1}_{a-1-n} \frac{1}{\sigma \varphi}(s) ds$. Then $(u_n)_{n\in \N}$ satisfies the desired properties and determines recurrence.\\
\centerline{}
(iii) Let
$$
u_n(x):= \begin{cases}\;1 \qquad \qquad \qquad \qquad &x\in [c_{-n},c_n],\\
\; 1-\frac{1}{a_n}\int_{c_n}^x \frac{1}{\sigma\varphi}(t) dt&x\in[c_n,d_n],\\
\; 1-\frac{1}{b_n}\int_x^{c_{-n}} \frac{1}{\sigma\varphi}(t) dt&x\in[d_{-n},c_{-n}], \\
\;0&\quad elsewhere,
\end{cases}
$$
where $a_n= \int_{c_n}^{d_n} \frac{1}{\sigma \varphi}(s) ds$, $b_n= \int_{c_{-n}}^{d_{-n}} \frac{1}{\sigma \varphi}(s) ds$. Then $(u_n)_{n\in \N}$ satisfies the desired properties and determines recurrence.\\
\centerline{}
(iv) and (v) are combinations of (ii) and (iii) and are proved by combining the proofs of (ii) and (iii).
\end{proof}
\centerline{}
\begin{rem}\label{notirreducible}
Note that we do not assume that $(\E,D(\E))$ is irreducible. As a non-trivial example consider the following:
Let $S=\{ x_i \in \R |\, i\in \Z\}$ with $x_i < x_{i+1}$ for all $i\in \Z$ and assume $S$ does not have an accumulation point in $\R$. For $\alpha \geq 1$, define a function $\varphi$ by
$$ 
\varphi(x)= |x-x_i|^\alpha,  x\in \big[ \frac{x_i+x_{i-1}}{2}, \frac{x_i+x_{i+1}}{2}   \big], i\in \Z.$$
Then $\varphi>0$ on $\R\setminus S$, hence $dx$-a.e. Assume $\sigma \equiv 1$, then since $1/\varphi \in L^1_{loc}(\R \setminus S,dx),$ (\ref{Hamza}) is also satisfied. Thus, the symmetric bilinear form $(\E,C_0^\infty(\R))$ defined by
$$
\E(f,g):=\frac{1}{2} \int_\R f'(x)g'(x) \mu(dx), \qquad f,g\in C^\infty_0(\R)
$$
is closable on $L^2(\R,d\mu)$ where $d\mu=\varphi dx$. Define the sequences $a_n$, $b_n$, $c_n$ and $d_n$ as in the Theorem \ref{nonreflected} (iii), then
$$
a_n= \int_{c_n}^{d_n} \frac{1}{\varphi}(s) ds=\int_{c_n}^{d_n} \frac{1}{(x_{n+1}-s)^\alpha} ds.
$$
(i) If $\alpha=1$, then
\begin{align*}a_n&=-\log(x_{n+1}-s)|^{d_n}_{c_n} \\
&=-\log(x_{n+1}-d_n) + \log (x_{n+1}-c_n).
\end{align*}
In this case, if $x_{n+1}-c_n>1$, then $x_{n+1}-d_n=\frac{1}{n}$ and
\begin{align*}a_n&=-\log\big(\frac{1}{n}\big) +\log(x_{n+1}-c_n) \\
&>\log n.
\end{align*}
And if $x_{n+1}-c_n\leq 1$, then $x_{n+1}-d_n=\frac{x_{n+1}-c_n}{n}$ and
\begin{align*}a_n&=-\log\big(\frac{x_{n+1}-c_n}{n}\big) +\log(x_{n+1}-c_n) \\
&=\log n.
\end{align*}
Thus $\lim_{n\rightarrow \infty}a_n=\infty$ if $\alpha=1$. \\ \\
(ii) If $\alpha>1$, then
\begin{align*}a_n&=\int_{c_n}^{d_n} \frac{1}{(x_{n+1}-s)^\alpha} ds \\
&=\frac{-1}{1-\alpha}(x_{n+1}-s)^{1-\alpha}|^{d_n}_{c_n} \\
&=\frac{1}{\alpha-1}\Big[   (x_{n+1}-d_n)^{1-\alpha}-(x_{n+1}-c_n)^{1-\alpha} \Big] .
\end{align*}
In this case,  if $x_{n+1}-c_n>1$, then $x_{n+1}-d_n=\frac{1}{n}$ and
\begin{align*}a_n&=  \frac{1}{\alpha-1}\Big[  \big (\frac{1}{n}\big)^{1-\alpha}-(x_{n+1}-c_n)^{1-\alpha} \Big]  \\
&>\frac{1}{\alpha-1}(n^{\alpha-1}-1).
\end{align*}
And if $x_{n+1}-c_n\leq 1$, then $x_{n+1}-d_n=\frac{x_{n+1}-c_n}{n}$ and
\begin{align*}a_n&=  \frac{1}{\alpha-1}\Big[   \big(\frac{x_{n+1}-c_n}{n}\big)^{1-\alpha}-(x_{n+1}-c_n)^{1-\alpha} \Big]  \\
&=  \frac{1}{\alpha-1}\Big[   (x_{n+1}-c_n)^{1-\alpha}  \big\{  \big(\frac{1}{n}\big)^{1-\alpha} -1 \big\} \Big]  \\
&\geq \frac{1}{\alpha-1} (n^{\alpha-1}-1).
\end{align*}
Thus $\lim_{n\rightarrow \infty}a_n=\infty$ if $\alpha \geq 1$. In the same way, $\lim_{n\rightarrow \infty} b_n=\infty$. 
Therefore, $(\E,D(\E))$ is recurrent, thus in particular conservative (cf. \cite[Theorems 1.6.5 and 1.6.6]{FOT}). Since $(\E,D(\E))$ is also strongly local, the process associated to $(\E,D(\E))$ is a conservative diffusion (cf. \cite{FOT}). Moreover, since by \cite[Example 3.3.2]{FOT} Cap$(\{x_i\})=0,$ the sets $(x_i, x_{i+1})$ are all invariant, i.e.
$$ p_t 1_{(x_i,x_{i+1})}(x) =0,\quad x\notin(x_i,x_{i+1}),\;\forall i \in \Z$$
where $p_t$ is the transition semigroup of (the process associated to) the Dirichlet form $(\E,D(\E))$. But
$$ d\mu((x_i,x_{i+1}))\neq 0, \qquad d\mu(\R\setminus (x_i,x_{i+1}))\neq0 \quad \forall i\in \Z.$$
Therefore $(\E,D(\E))$ is not irreducible (in the sense of \cite{FOT}).
\end{rem}
\centerline{}

\section{The one dimensional reflected case}\label{3}
Let $I=[0,\infty)$ and $C_0^\infty(I):=\{f:I \rightarrow \R \; |\; \exists g \in C_0^\infty (\R)$ with $g=f$ on $I \}$. Let $\varphi \in L^1_{loc}(I,dx)$ with $\varphi>0$ $dx$-a.e. Furthermore assume $\sigma$ be a (measurable) function such that $\sigma>0$ $dx$-a.e. and such that $\sigma \varphi \in L^1_{loc}(I, dx)$. Consider the symmetric bilinear form
$$\E(f,g):= \frac{1}{2} \int_I \sigma(x) f'(x)g'(x)  \mu(dx), \qquad f,g\in C^\infty_0(I)$$
on $L^2 (I,\mu)$ where as before $\mu:=\varphi dx$. As in the Section \ref{2}, we suppose $U$ is the largest open set in $I$ such that
$$\frac{1}{\sigma\varphi}\in L^1_{loc}(U,dx)$$
and assume
\begin{equation}\label{Hamza2}
dx(I \setminus U)=0.
\end{equation}
\centerline{}
We suppose $(\E,C_0^\infty (I ))$ is closable on $L^2(I,\mu)$. (By the results of \cite[Lemma 1.1]{tr}, if there is some open set $\tilde{U}\subset I$ such that $1/\varphi \in L^1_{loc}(\tilde{U},dx)$ and $dx(I \setminus(U\cap \tilde{U}))=0$, then $(\E,C_0^\infty(I))$ is closable on $L^2(I,\mu)$). Denote the closure by $(\E,D(\E))$. \\
There are two possible cases that we summarize in the following theorem.

\begin{theo}\label{reflected}
$(\E,D(\E))$ is recurrent if one of the following conditions holds:\\
(i) $U= V \cupdot (a,\infty)$ where $V$ is some open set (so either $V$ is empty if $a=0$, or $V$ is non-empty if $a>0$) and
$$
\int^\infty_{a+1}\frac{1}{\sigma \varphi}(s) ds =  \infty.
$$
(ii) $U= V \cupdot \bigcup_{n\in \N}I_n$, $I_n=(x_n, x_{n+1})$, $x_n < x_{n+1}$, $n\in \N$, $V$ some open set, and
\begin{align*}
a_n&:= \int_{c_n}^{d_n} \frac{1}{\sigma \varphi}(s) ds \; \longrightarrow \infty
\end{align*}
as $n \rightarrow \infty$, where  for $n\in \N$
$$
c_n= \frac{x_{n+1}+x_n}{2}
$$
and 
$$
d_n:= \begin{cases}\;x_{n+1}-\frac{1}{n}\qquad \;\qquad &\;\;x_{n+1}-c_n>1,\\
\; x_{n+1}-\frac{x_{n+1}-c_n}{n} &\;\;x_{n+1}-c_n\leq 1.
\end{cases}
$$
\end{theo}
\centerline{}
\begin{proof}
(i) Let
$$
u_n(x):= \begin{cases}\;1 \qquad \qquad \qquad \qquad &x\in [0,a+1],\\
\; 1-\frac{1}{a_n}\int_{a+1}^x \frac{1}{\sigma \varphi}(t)dt&x\in[a+1,a+1+n],\\
\;0&\quad elsewhere,
\end{cases}
$$
where $a_n=\int^{a+1+n}_{a+1} \frac{1}{\sigma\varphi}(s) ds$. Then  
$(u_n)_{n\in \N}\subset D(\E)$ with $0\leq u_n \leq 1$, $n\in \N$, 
$u_n \nearrow 1 $ as $n \rightarrow \infty $ and $ \lim_{n\rightarrow \infty}\E(u_n , u_n)=0.$\\
(ii) Let
$$
u_n(x):= \begin{cases}\;1 \qquad \qquad \qquad \qquad &x\in [0,c_n],\\
\; 1-\frac{1}{a_n}\int_{c_n}^x \frac{1}{\sigma\varphi}(t) dt&x\in[c_n,d_n],\\
\;0&\quad elsewhere,
\end{cases}
$$
where $a_n= \int_{c_n}^{d_n} \frac{1}{\sigma \varphi}(s) ds$. Then $(u_n)_{n\in \N}$ satisfies the desired properties and determines recurrence.
\end{proof}
\centerline{}
\centerline{}
\begin{exam}\label{Bessel}
If $\varphi(x)=x^{\delta-1}$ with $\delta>0$ and $\sigma(x)\equiv 1$ on $I$, then clearly (\ref{Hamza2}) is satisfied. In this case the process associated to the regular Dirichlet form $(\E,D(\E))$ (cf. \cite{FOT}) is the well-known Bessel process of dimension $\delta$. We are going to find a sufficient condition on the dimension $\delta$ for recurrence. Since
$$
\int_1^\infty \frac{1}{\varphi}(s) ds =\int_1^{\infty} s^{1-\delta}ds =\begin{cases}\;\lim_{n\rightarrow \infty}\log\,n   \qquad \qquad &\delta=2,\\
\; \lim_{n\rightarrow \infty}\frac{1}{2-\delta}[n^{2-\delta}-1]&\delta \neq 2,\\
\end{cases}
$$
we see by Theorem \ref{reflected} (i) with $a=0$ that the Bessel processes of dimension $\delta$ is recurrent if $\delta\in(0,2]$. Note that using \cite[Theorem 3]{Stu1} we obtain the same calculations up to a constant. However, in \cite{Stu1} the Dirichlet form is supposed to be irreducible throughout which we do not demand.
\end{exam}
\begin{rem}\label{closedset}
Of course Theorem \ref{reflected} can be easily reformulated for Dirichlet forms defined on more general closed sets (cf. Remark \ref{openset}).
\end{rem}
\centerline{}

\section{The multidimensional case}\label{4}
We have found sufficient conditions for recurrence of Dirichlet forms ($\E,D(\E))$ with one-dimensional state space. Now we will extend our results to multi-dimensional Dirichlet forms of gradient type.\\
Let $d\geq 2$ and $\varphi \in L_{loc}^1(\R^d,dx)$ such that $\varphi>0$ $dx$-a.e. Let further $a_{ij}$  be measurable functions such that $a_{ij}=a_{ji}$, $1\leq i,j \leq d$ and such that for each compact set $K\subset \R^d$, there exists $c_K >0$ such that
\begin{equation}\label{locallyelliptic}
\sum_{i,j=1}^d a_{ij}(x) \xi_i \xi_j \leq c_K \sum_{i=1}^d \xi_i ^2, \qquad \forall x\in K,\quad \forall\xi\in\R^d.
\end{equation}
Assume that the bilinear form
$$\E(f,g)=\sum_{i,j=1}^d\int_{\R^d}  a_{ij}(x)\partial_i f(x) \partial_j g(x)  \varphi(x) dx, \quad f,g \in C_0^\infty(\R^d)$$
is well-defined, positive and closable on $L^2(\R^d,\mu)$  where $d\mu=\varphi dx$. By (\ref{locallyelliptic}), we can define an increasing function $b(r)$ on $[0,\infty)$ by
$$
b(r):=c_{\overline{B_r(0)}}
$$
where $\overline{B_r(0)}:=\{x\in \R^d\,|\,|x|\leq r\}$, and $|\cdot |$ denotes the Euclidean norm. Let $S$ be the $(d-1)$-dimensional surface measure. Define
$$\psi(r):=\int_{\partial \overline{B_r(0)}} \varphi (x) S(dx),$$ and assume $\psi(r) \in (0,\infty)$ $dr$-a.e. Let $U\subset [0,\infty)$ be the largest open set such that
$$
\frac{1}{\psi}\in L^1_{loc}(U,dx).
$$
We suppose $dx([0,\infty) \setminus U)=0$.
\centerline{}
\begin{rem} 
Note that $\psi(r)>0$ for almost every $r>0$ follows easily from \cite[Theorem, page 38]{maz} since $\varphi>0$ $dx$-a.e. 
Hence our assumption $\psi(r) \in (0,\infty)$ for a.e. $r$ is satisfied whenever $\varphi\in L^1(\partial \overline{B_r(0)},dS)$ for a.e. $r$. The latter is for instance satisfied if $\varphi \in H^{1,1}_{loc}(\R^d,dx)$.
\end{rem}
\centerline{}
Again we will present sufficient conditions for the recurrence of $(\E,D(\E))$, i.e. for the existence of $(u_n)_{n\in \N}\subset D(\E)$ with $0\leq u_n \leq 1$, $n\in \N$, $u_n \nearrow 1 $ as $n \rightarrow \infty $ and $ \lim_{n\rightarrow \infty}\E(u_n , u_n)=0.$

\centerline{}

\begin{theo}\label{multidim1}
$(\E,D(\E))$ is recurrent if one of the following conditions holds: \\
(i) $U= V \cupdot (a,\infty)$ where $V$ is some open set (so either $V$ is empty if $a=0$, or $V$ is non-empty if $a>0$) and
$$a_n:=\int_{a+1}^{a+1+n}\frac{1}{\psi}(s) ds \longrightarrow \infty,\quad  \frac{b(a+1+n)}{a_n} \longrightarrow 0$$
as $n \rightarrow \infty.$\\
(ii) $U= V \cupdot \bigcup_{n\in \N}I_n$, $I_n=(x_n, x_{n+1})$, $x_n < x_{n+1}$, $n\in \N$, $V$ some open set, and
$$a_n:= \int_{c_n}^{d_n} \frac{1}{\psi}(s) ds \; \longrightarrow \infty,\quad \frac{b(d_n)}{a_n}\longrightarrow 0$$
as $n \rightarrow \infty$, where for $n\in \N$
$$c_n= \frac{x_{n+1}+x_n}{2}$$
and
$$d_n:= \begin{cases}\;x_{n+1}-\frac{1}{n}\qquad \;\qquad &\;\;x_{n+1}-c_n>1,\\
\; x_{n+1}-\frac{x_{n+1}-c_n}{n} &\;\;x_{n+1}-c_n\leq 1.
\end{cases}$$
\end{theo}
\centerline{}
\begin{proof}
(i) Let
$$u_n(x):= \begin{cases}\;1 \qquad \qquad \qquad \qquad &|x|\in [0,a+1],\\
\; 1-\frac{1}{a_n}\int_{a+1}^{|x|} \frac{1}{\psi}(t)dt&|x|\in[a+1,a+1+n],\\
\;0&\quad elsewhere,
\end{cases}$$
where $a_n=\int^{a+1+n}_{a+1} \frac{1}{\psi}(s) ds$. Then  $(u_n)_{n\in \N}\subset D(\E)$ with $0\leq u_n \leq 1$, $n\in \N$, $u_n \nearrow 1 $ as $n \rightarrow \infty.$ Note that
\begin{align*}
\E(u_n,u_n)&=\sum_{i,j=1}^d\int_{\R^d}  a_{ij}(x)\partial_i u_n(x) \partial_j u_n(x)  \varphi(x) dx\\
&= \int_{\overline{B_{a+1+n}(0)} \setminus B_{a+1}(0)} \sum_{i,j=1}^d a_{ij}(x)x_i x_j \frac{1}{|x|^2}\frac{1}{a_n^2} \frac{1}{\psi(|x|)^2} \varphi(x)dx\\
&\leq \frac{b(a+1+n)}{a_n^2} \int_{a+1}^{a+1+n}\,\int_{\partial\overline{B_r(0)}} \varphi(x) S(dx)\,\frac{1}{\psi(r)^2}dr\\
&=\frac{b(a+1+n)}{a_n}.
\end{align*}
where $B_r(0):=\{x\in \R^d\,|\,|x|<r\}$. Therefore $ \lim_{n\rightarrow \infty}\E(u_n , u_n)=0.$\\
(ii) Let
$$u_n(x):= \begin{cases}\;1 \qquad \qquad \qquad \qquad &|x|\in [0,c_n],\\
\; 1-\frac{1}{a_n}\int_{c_n}^{|x|} \frac{1}{\psi}(t) dt&|x|\in[c_n,d_n],\\
\;0&\quad elsewhere,
\end{cases}$$
where $a_n= \int_{c_n}^{d_n} \frac{1}{\psi}(s) ds$. Then $(u_n)_{n\in \N}$ satisfies desired the properties and determines recurrence.
\end{proof}
\centerline{}
Finally, let us consider a case which is possibly easier to calculate than the previous one. Let $\varphi \in L^1_{loc}(\R^d,dx)$, $\varphi>0$ $dx$-a.e. and  $a_{ij}$  be measurable functions such that $a_{ij}=a_{ji}$, $1\leq i,j \leq d$. We assume that the bilinear form
$$\E(f,g)=\sum_{i,j=1}^d\int_{\R^d}  a_{ij}(x)\partial_i f(x) \partial_j g(x)  \mu(dx), \quad f,g \in C_0^\infty(\R^d)$$
is well-defined, positive and closable on $L^2(\R^d,\mu)$ where $\mu:=\varphi dx$. \\
\centerline{}
\begin{theo}\label{multidim2}
Suppose that there is some compact set $K\subset \R^d$ and a function $\phi$ with
$$
\|A(x)\|\varphi(x) \leq \phi(| x |) \qquad \forall x\in \R^d\setminus K
$$
where $\|A(x)\|= \big[ \sum_{i,j=1}^d a_{ij}(x)^2   \big]^{\frac{1}{2}}$ and  $\phi \in L^1_{loc}(\R^d \setminus K,dx)$. Let
$$
a_n:= \int_{\overline{B_n(0)}\setminus B_\rho(0)} \frac{|y|^{2-2d}}{\phi (|y|)} dy
$$
where $\rho>0$ is such that $K\subset B_\rho (0)$ and $n>\rho$. If $a_n$ is finite for any $n>\rho$ and converges to $\infty$ as $n\rightarrow \infty$, then $(\E,D(\E))$ is recurrent.
\end{theo}
\begin{proof}
Since $\frac{|y|^{2-2d}}{\phi (|y|)}$ is rotationally invariant, we can rewrite $a_n$ as
$$
a_n=d\cdot vol_d(B_1(0)) \int_\rho^n \frac{s^{2-2d}}{\phi(s)}s^{d-1}ds =d\cdot vol_d(B_1(0)) \int_\rho^n \frac{s^{1-d}}{\phi(s)}ds
$$
where $vol_d(B_1(0))$ is the  volume of $B_1(0)$ in $\R^d$.
Define
$$ 
\psi_n(r):= \begin{cases}\;1  \qquad \qquad \qquad \qquad \qquad \qquad  &r\in [0,\rho],\\
\; 1-\frac{1}{a_n}\int_{\overline{B_r(0)}\setminus B_\rho(0)} \frac{|y|^{2-2d}}{\phi (|y|)} dy  &r \in [\rho,n], \\
\; 0  &elsewhere.
\end{cases}
$$
Then, we can rewrite $\psi_n$ as
$$ 
\psi_n(r)= \begin{cases}\;1  \qquad \qquad \qquad \qquad \qquad \qquad  &r\in [0,\rho],\\
\; 1-\frac{d\cdot vol_d(B_1(0))}{a_n}\int_\rho^r \frac{s^{1-d}}{\phi(s)}ds  &r \in [\rho,n],\\
\; 0  &elsewhere.
\end{cases}
$$
Set $u_n(x)=\psi_n(|x|)$, $n\in \N$ sufficiently large. Then $u_n$ is rotationally invariant, differentiable $dx$-a.e. and in $D(\E)$. Note that
\begin{align*} 
\E(u_n,u_n)&=\sum_{i,j=1}^d\int_{\R^d}  a_{ij}(x)\partial_i u_n(x) \partial_j u_n(x)  \mu(dx)\\
&= \sum_{i,j=1}^d\int_{\overline{B_n(0)}\setminus B_\rho(0)} \frac{d^2\cdot vol_d(B_1(0))^2}{a_n^2}a_{ij}(x)x_i x_j|x|^{-2d} \frac{1}{\phi^2 (|x|)} \varphi (x) dx \\
&\leq \frac{d^2\cdot vol_d(B_1(0))^2 }{a_n^2} \int_{\overline{B_n(0)}\setminus B_\rho(0)}   |x|^{2-2d} \frac{\|A(x)\|}{\phi^2 (|x|)} \varphi (x) dx \\
&\leq  \frac{d^2\cdot vol_d(B_1(0))^2}{a_n^2}\int_{\overline{B_n(0)}\setminus B_\rho(0)} |x|^{2-2d} \frac{1}{\phi (|x|)}  dx \\
&= \frac{d^2\cdot vol_d(B_1(0))^2}{a_n}.
\end{align*}
Since $d\cdot vol_d(B_1(0))$ is independent on $n$,  $\E(u_n,u_n) \longrightarrow 0$ as $n \rightarrow \infty$. Therefore $(\E,D(\E))$ is recurrent.
\end{proof}
\begin{rem}
Of course, Theorems \ref{multidim1} and \ref{multidim2} can be easily reformulated for more general open sets and in the reflected case (cf. Remarks \ref{openset} and \ref{closedset}, and Section \ref{3}).
\end{rem}


\begin{thebibliography}{XXX}


\bibitem{FOT} Fukushima, M. , Oshima, Y., Takeda, M.: Dirichlet forms and symmetric Markov processes, Walter de Gruyter 1994. x+392 pp.

\bibitem{FOT2} Fukushima, M. , Oshima, Y., Takeda, M.: Dirichlet forms and
Symmetric Markov processes. Berlin-New York: Walter de Gruyter 2011.

\bibitem{ge80} Getoor, R. K.: Transience and recurrence of Markov processes, Seminar on Probability, XIV, pp. 397--409, Lecture Notes in Math., 784, Springer, Berlin, 1980.

\bibitem{mr} Ma, Z.M., R\"ockner, M.: Introduction to the Theory of (Non-Symmetric)
 Dirichlet Forms. Berlin: Springer 1992.

\bibitem{maz} Maz'ya, V.: Sobolev spaces with applications to elliptic partial differential equations. Second, revised and augmented edition. Grundlehren der Mathematischen Wissenschaften, Vol. 342. Springer, Heidelberg, 2011. 

\bibitem{oukrutr} Ouknine, Y., Russo, F., Trutnau, G.: On countably skewed Brownian motion with accumulation point, arXiv, 2013.

\bibitem{o92r} Oshima, Y.: On conservativeness and recurrence criteria for Markov processes, Potential Anal. 1 (1992), no. 2, 115--131.

\bibitem{Stu1} Sturm. K. T.: Analysis on local Dirichlet spaces. I. Recurrence, conservativeness and $L^p$-Liouville properties, J. Reine Angew. Math. 456 (1994), 173-196.
\bibitem{tr} Trutnau, G.: Skorokhod decomposition of reflected diffusions on bounded Lischitz domains with singular non reflection part, Probab. Theory Relat. Fields 127 (2003), No.4, pp. 455-495.
\end{thebibliography}
\end{document}